\documentclass[a4paper,10pt]{amsart}%
\usepackage{eurosym}
\usepackage{amsmath}
\usepackage{mathrsfs}
\usepackage{amsfonts}
\usepackage{graphicx}
\usepackage{color}
\usepackage{amsfonts}
\usepackage{amssymb}%
\usepackage{esint}

\usepackage{palatino}

\usepackage[abbrev, backrefs]{amsrefs}
\usepackage{mathtools}
\usepackage{tikz-cd}
\usepackage{float}

\newtheorem{theorem}{Theorem}[section]

\newtheorem{corollary}[theorem]{Corollary}

\newtheorem{lemma}[theorem]{Lemma}

\newtheorem{remark}[theorem]{Remark}

\numberwithin{equation}{section}

\usepackage{combelow}

\newcommand{\R}{\mathbb{R}}

\newcommand{\barint}{
\rule[.036in]{.12in}{.009in}\kern-.16in \displaystyle\int }

\newcommand{\barcal}{\mbox{$ \rule[.036in]{.11in}{.007in}\kern-.128in\int $}}

\usepackage{enumitem}
\makeatletter
\let\@wraptoccontribs\wraptoccontribs
\makeatother

\mathchardef\mhyphen="2D

\title{Capacitary Maximal Inequalities and Applications}

\author[Y.-W. Chen]{You-Wei Benson Chen}
\author[K. H. Ooi]{Keng Hao Ooi}
\author[D. Spector]{Daniel Spector}

\address[Y.-W. Chen]{
National Chiao Tung University, Department of Applied Mathematics, 1001 Ta Hsueh Rd, 30010 Hsinchu, Taiwan, R.O.C.
}
\email{bensonchen.sc07@nycu.edu.tw}

\address[K. H. Ooi]{Department of Mathematics,
National Central University,
No.300, Jhongda Rd., Jhongli City, Taoyuan County 32001, Taiwan (R.O.C.).}
\email{kooi1@math.ncu.edu.tw}

\address[D. Spector]{Department of Mathematics, National Taiwan Normal University, No. 88, Section 4, Tingzhou Road, Wenshan District, Taipei City, Taiwan 116, R.O.C.
}
\email{spectda@protonmail.com}

\date{}

\begin{document}

\maketitle

\begin{abstract}
In this paper we introduce capacitary analogues of the Hardy-Littlewood maximal function,
\begin{align*}
\mathcal{M}_Cf(x):= \sup_{r>0} \frac{1}{C(B(x,r))} \int_{B(x,r)} |f|\;dC,
\end{align*}
for $C=$ the Hausdorff content or a Riesz capacity.  For these maximal functions, we prove a strong-type $(p,p)$ bound for $1<p \leq+\infty$ on the capacitary integration spaces $L^p(C)$ and a weak-type $(1,1)$ bound on the capacitary integration space $L^1(C)$.  We show how these estimates clarify and improve the existing literature concerning maximal function estimates on capacitary integration spaces.  As a consequence, we deduce correspondingly stronger differentiation theorems of Lebesgue-type, which in turn, by classical capacitary inequalities, yield more precise estimates concerning Lebesgue points for functions in Sobolev spaces.  
\end{abstract}

\section{Introduction}
Let $f:\mathbb{R}^{n} \to \mathbb{R}$ be a Lebesgue almost everywhere defined locally integrable function and denote by
\begin{align*}
\mathcal{M}f(x)=\sup_{r>0}\frac{1}{|B(x,r)|}\int_{B(x,r)}|f(y)|dy,\quad x\in\mathbb{R}^{n}
\end{align*}
its Hardy-Littlewood maximal function.  It is a classical result \cite[Theorem 1 on p.~5]{S} that the Hardy-Littlewood maximal operator $\mathcal{M}$ is of strong-type $(p,p)$ for $1<p \leq+\infty$ and weak-type $(1,1)$:  For $1<p<+\infty$, there exists a constant $C=C(n,p)>0$ such that
    \begin{align}\label{strong_HL}
      \int_{\mathbb{R}^n} (\mathcal{M} f)^p \; dx \leq  
 C \int_{\mathbb{R}^n}  |f|^p \; dx,
    \end{align}
    for all $f \in L^p(\mathbb{R}^n)$; when $p=+\infty$ one has a corresponding inequality involving the essential supremum; for $p=1$ there exists a constant $C=C(n)>0$ such that
   \begin{align}\label{weak_endpoint_HL}
        \left| \{ x\in \mathbb{R}^n: \mathcal{M} f(x) >t  \}   \right| \leq \frac{C}{t} \int_{\mathbb{R}^n}  |f| \; d x
   \end{align}
for all $f \in L^1(\mathbb{R}^n)$.  As is well-known, the inequality \eqref{strong_HL} follows from interpolation of the $L^\infty(\mathbb{R}^n)$ estimate and \eqref{weak_endpoint_HL}.  

The inequalities \eqref{strong_HL} and \eqref{weak_endpoint_HL}, along with density of continuous, compactly supported functions in $L^p(\mathbb{R}^n)$, imply by standard arguments \cite[p.~11]{S} that for $f \in L^p(\mathbb{R}^n)$ one has the existence of the precise representative
\begin{align}\label{representative}
f^*(x):= \lim_{r \to 0} \frac{1}{|B(x,r)|}\int_{B(x,r)}f(y)dy
\end{align}
and the convergence
\begin{align}\label{convergence}
\lim_{r \to 0} \frac{1}{|B(x,r)|}\int_{B(x,r)}|f(y)-f^*(x)|^p \;dy = 0
\end{align}
for Lebesgue almost every $x \in \mathbb{R}^n$.

While the strong-type $(1,1)$ inequality - \eqref{strong_HL} with $p=1$ - is known to fail, D.R. Adams \cite{AdamsChoquet} proved that for capacitary integration spaces defined in terms of the Hausdorff content one has such a strong-type estimate for this exponent.  The proof of this result was later simplified by J. Orobitg and J. Verdera \cite{OV}, who somewhat surprisingly showed that one can even obtain a strong-type inequality for exponents less than one, along with a weak-type estimate at an endpoint (below which the estimate fails).  To state their results, we recall that for $\beta \in (0,n]$ the Hausdorff content $\mathcal{H}^{\beta}_{\infty}$ is the set function defined by
\begin{align*}
\mathcal{H}^{\beta}_{\infty}(E):= \inf \left\{\sum_{i=1}^\infty \omega_{\beta} r_i^\beta : E \subset \bigcup_{i=1}^\infty B(x_i,r_i) \right\},
\end{align*}
where $\omega_\beta:= \pi^{\beta/2}/ \Gamma(\beta/2+1)$ is a normalization constant such that for $k \in \mathbb{N}$, $\mathcal{H}^{k}_{\infty}$ restricted to $k$-dimensional hyperplanes is the $k$-dimensional Lebesgue measure.  

Let $\alpha \in [0,n)$ so that $n-\alpha \in (0,n]$.  For an $\mathcal{H}^{n-\alpha}_{\infty}$ almost everywhere defined non-negative function $g : \mathbb{R}^n \to [0,\infty]$ we define the Choquet integral of $g$ with respect to $\mathcal{H}^{n-\alpha}_{\infty}$ by
\begin{align}\label{choquet_integral}
\int_{\mathbb{R}^n}  g \; d \mathcal{H}^{n-\alpha}_\infty := \int_0^\infty \mathcal{H}^{n-\alpha}_\infty(\{ g>t\})\;dt.
\end{align}
Note that monotonicity of the set function $\mathcal{H}^{n-\alpha}_\infty$ implies that this is always well-defined as a Lebesgue integral.  For $p \in (0,\infty)$, we define the capacitary integration spaces $L^p(\mathcal{H}^{n-\alpha}_{\infty})$ as the set of 
$\mathcal{H}^{n-\alpha}_{\infty}$-quasicontinuous functions (see \cite[Section 2]{Chen-Spector} or \cite[Section 3, p.~15]{AdamsChoquet}, for example) for which the quasi-norm
\begin{align}\label{quasi-norm}
\|f\|_{L^p(\mathcal{H}^{n-\alpha}_{\infty})}:=\left(\int_{\mathbb{R}^n}  |f|^p \; d \mathcal{H}^{n-\alpha}_\infty\right)^{1/p}
\end{align}
is finite.

The result of Adams-Orobitg-Verdera can now be stated as
\begin{theorem}\label{basetheorem}
    Let $\alpha \in [0,n)$, $ p \in [1-\alpha/n  ,+\infty)$, and suppose $f \in L^p(\mathcal{H}^{n-\alpha}_\infty)$.
    \begin{itemize}
    \item[(i)]
 If $p>1-\frac{\alpha}{n }$, then there exists a constant $C=C(\alpha, n,p)>0$ such that 
    \begin{align*}
      \int_{\mathbb{R}^n} (\mathcal{M} f)^p \; d \mathcal{H}^{n-\alpha}_\infty \leq  
 C \int_{\mathbb{R}^n}  |f|^p \; d \mathcal{H}^{n-\alpha}_\infty.
    \end{align*}
\item[(ii)] If $p = 1-\frac{\alpha}{n}$, then there exists a constant $C=C(\alpha,n)>0$ such that 
   \begin{align*}
       \mathcal{H}^{n-\alpha}_\infty \left( \{ x\in \mathbb{R}^n: \mathcal{M} f(x) >t  \}   \right)\leq \frac{C}{t^{p}} \int_{\mathbb{R}^n}  |f|^p \; d \mathcal{H}^{n-\alpha}_\infty.
   \end{align*}
   \end{itemize}
   \end{theorem}
 
 It is relevant to note that the authors in \cite{OV} do not require $f \in L^p(\mathcal{H}^{n-\alpha}_\infty)$, only the weaker condition that $f$ is $\mathcal{H}^{n-\alpha}_{\infty}$ almost everywhere defined with $\|f\|_{L^p(\mathcal{H}^{n-\alpha}_{\infty})}<+\infty$.  With the Lebesgue measure there is no distinction, because one automatically has density of continuous functions with compact support in $L^p(\mathbb{R}^n)$.  In the case of the Hausdorff content $\mathcal{H}^{n-\alpha}_{\infty}$, however, one has numerous examples of functions which are everywhere defined with finite integral and yet not possible to approximate by continuous compactly supported functions, for example, the characteristic function of any set when $\alpha \geq 1$!  Our choice to follow the assumptions of \cite{AdamsChoquet} is because when one works with the space $L^p(\mathcal{H}^{n-\alpha}_{\infty})$, the density of continuous, compactly supported functions implies by the same argument as in the $L^p(\mathbb{R}^n)$ setting that one has the existence of \eqref{representative} and the convergence \eqref{convergence} for $\mathcal{H}^{n-\alpha}$ almost every $x \in \mathbb{R}^n$.  In this way, one can deduce an improvement to the usual Lebesgue differentiation theorem for functions in the various Sobolev spaces which embed into $L^1(\mathcal{H}^{n-\alpha}_{\infty})$ \cite{AdamsChoquet, PS, PS1} - the existence of the precise representative $\mathcal{H}^{n-\alpha}$ almost everywhere.
 
In this paper we obtain a further improvement of this Lebesgue differentiation theorem by a refinement of Theorem \ref{basetheorem} which has the additional benefit of clarifying the somewhat mysterious validity of strong-type estimates for exponents less than one.  To this end we define the centered Hausdorff content maximal function
\begin{align*} 
\mathcal{M}_{\mathcal{H}^\beta_\infty} f(x) \vcentcolon= \sup_{r>0}  \frac{1}{\omega_\beta r^\beta} \int_{B(x,r)} |f| \;\;d\mathcal{H}^{\beta}_{\infty}.
\end{align*}
We can now state the first result of this paper, an improvement of Theorem \ref{basetheorem} with a capacitary maximal function in place of the Hardy-Littlewood maximal function, our  
\begin{theorem}\label{maintheorem1}
    Let $\beta \in (0,n)$, $p \in [1,+\infty)$, and suppose $f \in L^p(\mathcal{H}^{\beta}_\infty)$. 
\begin{itemize}
\item[(i)] If $p>1$, then there exists a constant $C=C(\beta, n,p)>0$ such that 
    \begin{align*}
      \int_{\mathbb{R}^n} (\mathcal{M}_{\mathcal{H}^\beta_\infty} f)^p \; d \mathcal{H}^\beta_\infty \leq  
 C \int_{\mathbb{R}^n}  |f|^p \; d \mathcal{H}^\beta_\infty.
    \end{align*}
\item[(ii)] If $p = 1$, then there exists a constant $C=C(\beta, n)>0$ such that 
   \begin{align*}
       \mathcal{H}^\beta_\infty \left( \{ x\in \mathbb{R}^n: \mathcal{M}_{\mathcal{H}^\beta_\infty} f(x) >t  \}   \right)\leq \frac{C}{t} \int_{\mathbb{R}^n}  |f| \; d \mathcal{H}^\beta_\infty.
   \end{align*}
   \end{itemize}
\end{theorem}
As in the classical case, part $(i)$ in Theorem \ref{maintheorem1} follows by the usual interpolation argument between the estimate asserted in Part $(ii)$ and the easy $L^\infty(\mathcal{H}^\beta_\infty)$ estimate.  Part $(ii)$, in turn, follows from a similar weak-type estimate proved for its dyadic counterpart in \cite[Theorem 3.1]{Chen-Spector} and the fact that the level sets of the two are two-sided comparable, see Lemma \ref{equivoftwomaximal} below.  

Beyond its ease of demonstration, Theorem \ref{maintheorem1} is of interest to us for its implications concerning Lebesgue differentiation, as from it we deduce
\begin{corollary}\label{corollary_LDT}
Let $\beta \in (0,n)$, $p \in [1,+\infty)$, and suppose $f \in L^p(\mathcal{H}^{\beta}_\infty)$.  The function
\begin{align}\label{representative_cor}
f^*(x):= \lim_{r \to 0} \frac{1}{|B(x,r)|}\int_{B(x,r)}f(y)dy
\end{align}
is well-defined and one has the convergence
\begin{align}\label{convergence_cor}
\lim_{r \to 0} \frac{1}{\omega_\beta r^\beta}\int_{B(x,r)}|f-f^*(x)|^p \;d\mathcal{H}^\beta_\infty = 0
\end{align}
for $\mathcal{H}^\beta$ almost every $x \in \mathbb{R}^n$.
\end{corollary}
\noindent
In particular, as one has that $W^{k,1}(\mathbb{R}^n) \hookrightarrow L^1( \mathcal{H}^{n-k}_\infty)$ (which follows from density of smooth, compactly supported functions in the former space and \cite[Theorem B]{AdamsChoquet}), from this one concludes that $f^*$ exists and
\begin{align*}
\lim_{r \to 0} \frac{1}{\omega_{n-k} r^{n-k}}\int_{B(x,r)}|f-f^*(x)| \;d\mathcal{H}^{n-k}_\infty = 0
\end{align*}
$\mathcal{H}^{n-k}$ almost every $x \in \mathbb{R}^n$.  For $k=1$ this is an improvement to \cite[(3) on p.~184]{Federer}, as the convergence with respect to the capacity $\mathcal{H}^{n-k}_\infty$ implies even convergence in the Lorentz space $L^{n/(n-k),1}(\mathbb{R}^n)$, for example.  This also yields an improvement to the fine properties of Riesz potentials of elements in the Hardy space which are a consequence of Adams' embedding \cite{AdamsChoquet}, and of curl and divergence free functions $F \in L^1(\mathbb{R}^n;\mathbb{R}^n)$ by \cite{RSS}, which builds on the atomic decomposition established in \cite{SpectorHernandez2020} and exploited in \cite{SpectorHernandezRaita2021}.

Our claim that Theorem \ref{basetheorem} can be deduced from Theorem \ref{maintheorem1} is substantiated by the choice $\gamma=n-\alpha$ and $\beta = n$ in 
  \begin{corollary}\label{maintheorem2}
    Let $0< \gamma <  \beta  \leq n$, $p\in[ \gamma / \beta,+\infty)$, and suppose $f \in L^p(\mathcal{H}^{\gamma}_\infty)$.
    \begin{itemize}
\item[(i)] If $p>\frac{\gamma}{\beta }$, then there exists a constant $C=C(\gamma,\beta,n,p)>0$ such that 
    \begin{align*}
      \int_{\mathbb{R}^n} (\mathcal{M}_{\mathcal{H}^\beta_\infty} f)^p \; d \mathcal{H}^{\gamma}_\infty \leq  
 C \int_{\mathbb{R}^n}  |f|^p \; d \mathcal{H}^{\gamma}_\infty.
    \end{align*}
\item[(ii)] If $p =\frac{\gamma}{\beta}$, then there exists a constant $C=C(\gamma,\beta,n)>0$ such that 
   \begin{align*}
       \mathcal{H}^{\gamma}_\infty \left( \{ x\in \mathbb{R}^n: \mathcal{M}_{\mathcal{H}^\beta_\infty} f(x) >t  \}   \right) \leq \frac{C}{t^{p}} \int_{\mathbb{R}^n}  |f|^p \; d \mathcal{H}^{\gamma}_\infty.
   \end{align*}
   \end{itemize} 
   \end{corollary}
\noindent
Corollary \ref{maintheorem2} relies on Theorem \ref{maintheorem1} and Lemma \ref{maximaldimetionalchange} in Section \ref{preliminaries}, the latter of which finally explains the appearance of exponents below one, which are a consequence of the change between the maximal functions of various dimensions in relation to a fixed Hausdorff content.

This phenomena of estimates below the exponent one extends beyond the capacitary integration spaces defined in terms of the Hausdorff content.  Indeed, motivated by Theorem \ref{basetheorem}, the second author and N.C. Phuc \cite{OP} proved that for capacitary spaces defined in terms of Riesz capacities one also has a strong-type bound for exponents less than one.  To state their results we recall the Riesz capacities ${\rm cap}_{\alpha,s}(\cdot)$, defined for $\alpha \in (0,n)$ and $1<s<\infty$ by
\begin{align*}
\text{cap}_{\alpha,s}(E)=\inf\{\|\varphi\|_{L^{s}(\mathbb{R}^{n})}^{s}:\varphi\geq 0,~I_{\alpha}\ast\varphi\geq 1~\text{on}~E\},\quad E\subseteq\mathbb{R}^{n},
\end{align*}
where 
\begin{align*}
I_{\alpha}(x)=\frac{1}{\gamma(\alpha)} \frac{1}{|x|^{n-\alpha}}, \quad x\in\mathbb{R}^{n},
\end{align*}
are the Riesz kernels, c.f. \cite[p.~117]{S}.  The Riesz capacities are also monotone, so that in analogy to \eqref{choquet_integral}, for an $\text{cap}_{\alpha,s}$ almost everywhere defined non-negative function $g : \mathbb{R}^n \to [0,\infty]$ we can define the Choquet integral with respect to the Riesz capacity $\text{cap}_{\alpha,s}$ by
\begin{align}\label{choquet_integral_prime}
\int_{\mathbb{R}^n}  g \; d\text{cap}_{\alpha,s} := \int_0^\infty \text{cap}_{\alpha,s}(\{ g>t\})\;dt.
\end{align}
For $0<p<+\infty$, the corresponding capacitary integration spaces are in turn defined as the set of $\text{cap}_{\alpha,s}$-quasicontinuous functions for which the quasi-norm
 \begin{align}\label{quasi-norm_prime}
\|f\|_{L^p(\text{cap}_{\alpha,s})}:=\left(\int_{\mathbb{R}^n}  |f|^p \; d \text{cap}_{\alpha,s}\right)^{1/p}
\end{align}
is finite.

With this preparation, Theorems 1.1 and 1.2 of \cite{OP} can be stated as
\begin{theorem}\label{basetheorem2}
Let $\alpha \in (0,n)$, $1<s<n/\alpha$, $p \in [1-\alpha s/n, +\infty)$, and suppose $f \in L^p(\text{cap}_{\alpha,s})$. 
\begin{itemize}
\item[(i)] If  $p>1-\frac{\alpha s}{n }$, then there exists a constant $C=C(\alpha,s,n,p)>0$ such that 
    \begin{align*}
      \int_{\mathbb{R}^{n}}\left(\mathcal{M}f\right)^{p}d{\rm cap}_{\alpha,s}\leq C\int_{\mathbb{R}^{n}}|f|^{p}d{\rm cap}_{\alpha,s}.
\end{align*}
\item[(ii)] If $p = 1- \frac{\alpha s}{n}$, then there exists a constant $C=C(\alpha,s,n)>0$ such that 
   \begin{align*}
      {\rm cap}_{\alpha,s}\left(\{x\in\mathbb{R}^{n}:\left(\mathcal{M}f(x)\right)^{p}>t\}\right)\leq \frac{C}{t^p}\int_{\mathbb{R}^{n}}|f|^{p}d{\rm cap}_{\alpha,s}.
   \end{align*}
   \end{itemize}
\end{theorem}

In light of the relation of Theorems \ref{basetheorem} and \ref{maintheorem1},  the consideration of Theorem \ref{basetheorem2} naturally suggests a strengthening in terms of a capacitary maximal function associated to the Riesz capacity.  A result in this direction has been established by D.R. Adams in \cite[p.~882]{Adams_PAMS_1988}, where he proves a weak-type estimate for the Newtonian capacity, i.e. $\alpha=1$ and $s=2$, when $p=2$.  The following result contains Adams estimate as a special case.
\begin{theorem}\label{main 1}
Let $\alpha \in (0,n)$, $1<s<n/\alpha$, $p \in [1,+\infty)$, and suppose $f \in L^p(\text{cap}_{\alpha,s})$. 
\begin{itemize}
\item[(i)] If  $p>1$, then there exists a constant $C=C(\alpha,s,n,p)>0$ such that 
  \begin{align*}
\int_{\mathbb{R}^{n}}\left(\mathcal{M}_{{\rm cap}_{\alpha,s}}f\right)^{p}d{\rm cap}_{\alpha,s}\leq C\int_{\mathbb{R}^{n}}|f|^{p}\;d{\rm cap}_{\alpha,s}.
\end{align*}
\item[(ii)] If $p = 1$, then there exists a constant $C=C(\alpha,s,n)>0$ such that 
 \begin{align*}
{\rm cap}_{\alpha,s}\left(\left\{x\in\mathbb{R}^{n}:\mathcal{M}_{{\rm cap}_{\alpha,s}}f(x)>t\right\}\right)\leq \frac{C}{t}\int_{\mathbb{R}^{n}}|f|\;d{\rm cap}_{\alpha,s}
\end{align*}
   \end{itemize}
\end{theorem}
\noindent
Here
\begin{align*}
\mathcal{M}_{{\rm cap}_{\alpha,s}}f(x)=\sup_{r>0}\frac{1}{{\rm cap}_{\alpha,s}(B(0,1))r^{n-\alpha s}}\int_{B(x,r)}|f|\;d{\rm cap}_{\alpha,s}.
\end{align*}
In \cite{Adams_PAMS_1988}, the demonstration of a weak-type estimate is to establish an improvement to the usual Lebesgue differentiation theorem for these spaces.  From Theorem \ref{main 1}  we obtain similar results for an extended range of parameters in
\begin{corollary}\label{corollary_LDT_Riesz}
Let $\alpha \in (0,n)$, $1<s<n/\alpha$, $p \in [1,+\infty)$, and suppose $f \in L^p(\text{cap}_{\alpha,s})$.  The function
\begin{align}\label{representative_cor_riesz}
f^*(x):= \lim_{r \to 0} \frac{1}{|B(x,r)|}\int_{B(x,r)}f(y)dy
\end{align}
is well-defined and one has the convergence
\begin{align}\label{convergence_cor_riesz}
\lim_{r \to 0} \frac{1}{{\rm cap}_{\alpha,s}(B(0,1))r^{n-\alpha s}}\int_{B(x,r)}|f-f^*(x)|^p \;d\text{cap}_{\alpha,s} = 0
\end{align}
for $\text{cap}_{\alpha,s}$ almost every $x \in \mathbb{R}^n$.
\end{corollary}
\noindent
For $\alpha =1$ and $s=p$, this improves \cite[9.~The main theorem on p.~153]{Federer-Ziemer}, as the convergence of the $p$-powers with respect to the capacity $\text{cap}_{1,p}$ implies even convergence in the Lorentz space $L^{np/(n-p),p}(\mathbb{R}^n)$, for example.

As with Theorems \ref{basetheorem} and \ref{maintheorem1}, one can deduce Theorem \ref{basetheorem2} from Theorem \ref{main 1} as a result of 
\begin{corollary}\label{corollary2}
    Let $0\leq \gamma <  \beta  \leq n$, $1<s<n/\alpha$, $p \in [(n-\alpha s)/(n-\gamma s), +\infty)$, and suppose $f \in L^p(\text{cap}_{\alpha,s})$.
    \begin{itemize}
\item[(i)] If $p>\frac{n-\alpha s }{n-\gamma s }$, then there exists a constant $C=C(\alpha,\gamma,s,n,p)>0$ such that 
    \begin{align*}
      \int_{\mathbb{R}^{n}}\left(\mathcal{M}_{{\rm cap}_{\gamma,s}}f\right)^{p}d{\rm cap}_{\alpha,s} \leq  
 C \int_{\mathbb{R}^{n}}|f|^p\;d{\rm cap}_{\alpha,s}.
    \end{align*}
\item[(ii)] If $p =\frac{n-\alpha s }{n-\gamma s }$, then there exists a constant $C=C(\alpha,\gamma,s,n)>0$ such that 
   \begin{align*}
       {\rm cap}_{\alpha,s}\left(\left\{x\in\mathbb{R}^{n}:\mathcal{M}_{{\rm cap}_{\gamma,s}}f(x)>t\right\}\right) \leq \frac{C}{t^{p}} \int_{\mathbb{R}^{n}}|f|^p\;d{\rm cap}_{\alpha,s}.
   \end{align*}
   \end{itemize} 
   \end{corollary}

\section{Preliminaries}\label{preliminaries}

A basic difficulty in working with capacitary integration is that the Choquet integral is not linear, and in fact, not even sublinear.  This requires some care in working with it and the associated capacitary integration spaces.  We refer the reader to the results \cite{A3, A25, C1, Cap2, C4, C11, Cerda, CerdaCollMartin,CerdaMartinSilvestre, H,H1, MS, PS,PS1,PS2, Saito,Saito1,ST,STW,STW1,Spector-PM,YangYuan} for the treatment of some technical aspects concerning capacities and application of these results which requires some delicacy.  We note that if one assumes that the capacity in consideration is monotone and strongly subadditive, then the associated Choquet integral is sublinear, and in fact, sublinearity is itself equivalent to this strong subadditivity (see \cite{PS2}, for example).  As the Hausdorff content $\mathcal{H}^{\beta}_{\infty}$ is not strongly subadditive for $\beta <n$, it is more useful to work with an equivalent set function which satisfies this strong subadditivity.   To this end we introduce the dyadic Hausdorff content associated to a cube $Q \subset \mathbb{R}^n$:
\begin{align*}
\mathcal{H}^{\beta,Q}_{\infty}(E):= \inf \left\{\sum_{i=1}^\infty l(Q_i)^\beta  : E \subset \bigcup_{i=1}^\infty Q_i, Q_i \in \mathcal{D}(Q) \right\},
\end{align*}
where $\mathcal{D}(Q)$ is the dyadic lattice generated by the cube $Q$.  These set functions are monotone and strongly subadditive, and therefore sublinear, from which we deduce
\begin{align}
     \int \sum^\infty_{j=1} f_j \;d \mathcal{H}^{\beta, Q}_\infty \leq \sum^\infty_{j=1}\int f_j \;d \mathcal{H}^{\beta,Q}_\infty,
\end{align}
see \cite[Proposition 3.5 and 3.6]{STW} for usual dyadic case and \cite[Proposition 2.6 and Proposition 2.10]{Chen-Spector} for that of a general cube $Q$. 

The equivalence of these dyadic Hausdorff contents and the Hausdorff content can be found in Proposition 2.3 in \cite{YangYuan} and Proposition 2.11 in \cite{Chen-Spector}, that there exists a constant $C_\beta>1$ such that for every $Q \subset \mathbb{R}^n$
\begin{align}\label{equivalenceofdyadic}
 \frac{1}{C_\beta} \mathcal{H}^{\beta,Q}_{\infty}(E) \leq  \mathcal{H}^{\beta}_{\infty}(E) \leq C_\beta \mathcal{H}^{\beta,Q}_{\infty}(E)
\end{align}
for all $E \subset \mathbb{R}^n$.  Therefore one finds that 
\begin{align}\label{sublinearofHbeta}
\int_{\mathbb{R}^n} \sum^\infty_{j=1} f_j \;d \mathcal{H}^\beta_\infty \leq     C_\beta^2 \sum^\infty_{j=1}\int_{\mathbb{R}^n} f_j \;d \mathcal{H}^\beta_\infty.
\end{align}


We next recall Theorem 3.1 in \cite{Chen-Spector}, the weak type $(1,1)$ estimate for dyadic maximal function with respect to the dyadic Hausdorff content, which is an easy consequence of a covering lemma of Mel\cprime nikov \cite{Mel}:  For any $t>0$ one has the inequality
\begin{align}\label{dyadicweaktypeestimate11}
    \mathcal{H}^{\beta,Q_0}_{\infty}\left(\left\{ \mathcal{M}^{\beta,Q_0}_\infty f>t\right\}\right) \leq \frac{C}{t} \int_{\mathbb{R}^n} |f|\;d\mathcal{H}^{\beta,Q_0}_{\infty}
\end{align}
for all $\mathcal{H}^{\beta}$ almost every where well-defined functions $f:\mathbb{R}^n \to \mathbb{R}$ .
\begin{remark}
Note that in \cite[Theorem 3.1]{Chen-Spector}, $f$ is assumed to be $\mathcal{H}^{\beta}_{\infty}$-quasicontinuous and satisfy 
\begin{align*}
    \int_{\mathbb{R}^n} |f| \;d \mathcal{H}^{\beta,Q_0}_\infty < \infty.  
\end{align*}
However, an examination of the proof shows these assumptions are not necessary for the weak type estimate (\ref{dyadicweaktypeestimate11}).
\end{remark}

The following result is an extension of \cite[Lemma 3]{OV}.
\begin{lemma}\label{alphabetachange}
    Let $f \geq 0$ and $0 < \gamma < \beta \leq n \in \mathbb{N}$. Then 
    \begin{align*}
        \int_{\mathbb{R}^n} f \; d\mathcal{H}^\beta_\infty \leq \omega_\beta  \left( \frac{1}{\omega_\gamma}\right)^{\frac{\beta}{\gamma}}   \frac{\beta }{\gamma} \left(  \int_{\mathbb{R}^n} f^{\frac{\gamma}{\beta}}  \; d\mathcal{H}^\gamma_\infty \right)^{\frac{\beta}{\gamma}}.
    \end{align*}
\end{lemma}
\begin{proof}
   Note that for every countable collection of open balls $\{ B(x_j, r_j)\}_j$ in $\mathbb{R}^n$, we have
\begin{align*}
    \left( \sum_j r_j^\beta  \right)^{\frac{1}{\beta}} \leq \left( \sum_j r_j^\gamma \right)^{\frac{1}{\gamma}}.
\end{align*}
 Hence, for every $E \subseteq \mathbb{R}^n$,
\begin{align*}
    \left(\frac{1}{\omega_{\beta} }   \mathcal{H}^\beta_\infty(E)\right) ^{\frac{1}{\beta}} \leq   \left(\frac{1}{\omega_{\gamma} }  \mathcal{H}^\gamma_\infty(E)\right)^{\frac{1}{\gamma}} 
\end{align*}
holds by the definition of the Hausdorff content. In particular, setting $E = \{ x \in \mathbb{R}^n: f(x)> t\}$, we obtain
    \begin{align*}
         \left(\frac{1}{\omega_{\beta} } \mathcal{H}^\beta_\infty(\{ x \in \mathbb{R}^n: f(x)> t\}) \right)^{\frac{1}{\beta}} \leq \left(\frac{1}{\omega_{\gamma} } \mathcal{H}^\gamma_\infty(\{ x \in \mathbb{R}^n: f(x)> t\})\right)^{\frac{1}{\gamma}} 
    \end{align*}
and so
\begin{align}\label{estimateofHbetatolevealset}
    \int_{\mathbb{R}^n} f \; d\mathcal{H}^\beta_\infty & = \int^\infty_0 \mathcal{H}^\beta_\infty( \{ x \in \mathbb{R}^n: f(x)> t\}) \;dt\\
    &\leq \omega_\beta  \left( \frac{1}{\omega_\gamma}\right)^{\frac{\beta}{\gamma}}  \int^\infty_0  \left( \mathcal{H}^\gamma_\infty (\{ x \in \mathbb{R}^n: f(x)> t\} ) \right)^{\frac{\beta} {\gamma}}\;dt \nonumber \\
    & = \omega_\beta  \left( \frac{1}{\omega_\gamma}\right)^{\frac{\beta}{\gamma}}  \frac{\beta}{\gamma} \int^\infty_0 t^{\frac{\beta}{\gamma} -1} \left( \mathcal{H}^\gamma_\infty (\{ x \in \mathbb{R}^n: f(x)> t^{\frac{\beta}{\gamma}}\} ) \right)^{\frac{\beta} {\gamma}}  \;dt.\nonumber
    \end{align}
    Applying Chebyshev's inequality, we obtain
\begin{align*}
    t \; \mathcal{H}^\gamma_\infty (\{ x \in \mathbb{R}^n: f(x)^{\frac{\gamma}{\beta}} > t\} ) \;  \leq  \int_{\mathbb{R}^n} f^{\frac{\gamma}{\beta}} \; d\mathcal{H}^\gamma_\infty
\end{align*}
and thus
\begin{align*}
    &t^{\frac{\beta}{\gamma} -1} \left( \mathcal{H}^\gamma_\infty (\{ x \in \mathbb{R}^n: f(x)> t^{\frac{\beta}{\gamma}}\} ) \right)^{\frac{\beta} {\gamma}} \\
    =& \left( t \; \mathcal{H}^\gamma_\infty (\{ x \in \mathbb{R}^n: f(x)^{\frac{\gamma}{\beta}} > t\} )  \right)^{\frac{\beta}{\gamma} -1} \mathcal{H}^\gamma_\infty (\{ x \in \mathbb{R}^n: f(x)^{\frac{\gamma}{\beta}} > t\} )\\
    \leq &  \left( \int_{\mathbb{R}^n} f^{\frac{\gamma}{\beta}} \; d\mathcal{H}^\gamma_\infty \right)^{\frac{\beta}{\gamma} -1}    \mathcal{H}^\gamma_\infty (\{ x \in \mathbb{R}^n: f(x)^{\frac{\gamma}{\beta}} > t\} )
\end{align*}
Combining this estimate with (\ref{estimateofHbetatolevealset}), we deduce that 
\begin{align*}
    \int_{\mathbb{R}^n} f \; d \mathcal{H}^\beta_\infty &\leq \omega_\beta  \left( \frac{1}{\omega_\gamma}\right)^{\frac{\beta}{\gamma}}  \frac{\beta}{\gamma}  \left( \int_{\mathbb{R}^n} f^{\frac{\gamma}{\beta}} \; d\mathcal{H}^\gamma_\infty \right)^{\frac{\beta}{\gamma} -1}  \int^\infty_0    \mathcal{H}^\gamma_\infty (\{ x \in \mathbb{R}^n: f(x)^{\frac{\gamma}{\beta}} > t\} ) \;dt\\
    & =  \omega_\beta  \left( \frac{1}{\omega_\gamma}\right)^{\frac{\beta}{\gamma}}   \frac{\beta}{\gamma}  \left( \int_{\mathbb{R}^n} f^{\frac{\gamma}{\beta}} \; d\mathcal{H}^\gamma_\infty \right)^{\frac{\beta}{\gamma} } .
 \end{align*}
    
\end{proof}

As a consequence of Lemma \ref{alphabetachange}, one can easily establish 
\begin{lemma}\label{maximaldimetionalchange}
    Let $f\geq 0$ and $0 < \gamma \leq \beta \leq n \in \mathbb{N}$. Then
    \begin{align}
        \mathcal{M}_{\mathcal{H}^\beta_\infty} f (x) \leq \frac{\beta}{\gamma} \left( \mathcal{M}_{\mathcal{H}^\gamma_\infty} \left( f^{\frac{\gamma}{\beta}} \right) (x) \right)^{\frac{\beta}{\gamma}}
    \end{align}
    holds for every $x \in \mathbb{R}^n$.
\end{lemma}
\begin{proof}
    Let $x \in \mathbb{R}^n$ and $r>0$. Then using Lemma \ref{alphabetachange}, we obtain 
    \begin{align*}
        \frac{1}{\omega_\beta r^\beta} \int_{B(x,r)} f \; d \mathcal{H}_\infty^\beta&\leq   \frac{\beta}{\gamma} \left(  
        \frac{1}{\omega_\gamma r^\gamma}  \int_{B(x,r)} f^{\frac{\gamma}{\beta}} \; d \mathcal{H}^\gamma_\infty \right)^\frac{\beta}{\gamma}\\
        &\leq \frac{\beta}{\gamma}\left(   
        \mathcal{M}_{\mathcal{H}^\gamma_\infty} (f ^{\frac{\gamma}{\beta}})(x)
        \right)^\frac{\beta}{\gamma}.
    \end{align*}
    Taking the supremum over all the balls with center $x$ in the preceding inequality we deduce that
    \begin{align*}
         \mathcal{M}_{\mathcal{H}^\beta_\infty} f (x) \leq \frac{\beta}{\gamma} \left( \mathcal{M}_{\mathcal{H}^\gamma_\infty} \left(f^{\frac{\gamma}{\beta}}\right) (x) \right)^{\frac{\beta}{\gamma}}.
    \end{align*}
\end{proof}

\begin{lemma}\label{equivoftwomaximal}
    Let $Q_0$ be a cube in $\mathbb{R}^n$. There exist constants $c,C>0$ such that
    \begin{align*}
        \mathcal{H}^\beta_\infty\left( \left\{  x \in \mathbb{R}^n : \mathcal{M}_{\mathcal{H}^\beta_\infty} f (x) >t \right\} \right) \leq C \mathcal{H}^\beta_\infty\left( \left\{  x\in \mathbb{R}^n : \mathcal{M}_\infty^{\beta, Q_0} f(x) > ct\right\} \right)
    \end{align*}
   holds for every non-negative function $f$.
\end{lemma}
\begin{proof}
Define the uncentered Hausdorff content maximal function
\begin{align*} 
\tilde{\mathcal{M}}_{\mathcal{H}^\beta_\infty} f(x) \vcentcolon= \sup_{x \in B}  \frac{1}{\mathcal{H}^\beta_\infty(B)} \int_{B} |f| \;\;d\mathcal{H}^{\beta}_{\infty}.
\end{align*}
Then as this maximal function dominates the centered Hausdorff content maximal function, it suffices to show
\begin{align*}
        \mathcal{H}^\beta_\infty\left( \left\{  x \in \mathbb{R}^n : \tilde{\mathcal{M}}_{\mathcal{H}^\beta_\infty} f (x) >t \right\} \right) \leq C \mathcal{H}^\beta_\infty\left( \left\{  x\in \mathbb{R}^n : \mathcal{M}_\infty^{\beta, Q_0} f(x) > c t\right\} \right).
    \end{align*}

    For $t>0$, we let 
    \begin{align*}
        E_t &: = \{ x \in \mathbb{R}^n : \tilde{\mathcal{M}}_{\mathcal{H}^\beta_\infty} f(x) >t \}, \\
        L_t&:= \{ x \in \mathbb{R}^n: \mathcal{M}_\infty^{\beta, Q_0}f (x) >t\}.
    \end{align*}
    We claim that there exists a universal constant $c=c(\beta,n)>0$ such that for every $x \in E_t$ one can find a ball $B^x \subseteq E_t$ and a dyadic cube $Q^x$ in $\mathcal{D}(Q_0)$ such that
    \begin{align}\label{TildeQcontaindinL}
    Q^x \subseteq L_{ c t} 
\end{align}
    and 
\begin{align*}
x \in B^x \subseteq 3 Q^x,
\end{align*}
    where $3 Q^x$ denotes the cube with the same center as $Q^x$, but with side length 3 times that of $Q^x$. 
  To prove the claim, we observe that for every $x \in E_t$, the definition of the uncentered Hausdorff content maximal function yields a ball $B^x$ such that 
    \begin{align*}
        \frac{1}{ \mathcal{H}^\beta_\infty(B^x) } \int_{B^x} |f| \; d \mathcal{H}^\beta_\infty >t.
    \end{align*}
 Here we note that it is possible that $\int_{B^x} |f| \; d \mathcal{H}^\beta_\infty = \infty$, which we allow in what follows, adopting the convention that $\alpha + \infty = \infty $ and $\alpha \cdot \infty = \infty$ if $0 <\alpha \leq \infty$.
 
    Let $r_x$ denote the radius of the ball $B^x$, so that $ \mathcal{H}^\beta_\infty(B^x) = \omega_{\beta}r_x^\beta$.  For each such a ball $B^x$, there exist at most $2^n$ dyadic cubes which we enumerate as $\{ Q_k^x \}_k $ in $\mathcal{D}(Q_0)$ such that
     \begin{align}\label{equivalentsize}
       \frac{\ell(Q_k^x)}{2} < 2r_x \leq \ell(Q_k^x) \text{ for each $k$, }
    \end{align}

    \begin{align}\label{nondisjoint}
        B^x \cap Q_k^x \neq \emptyset \text{ for each $k$, }
    \end{align}
and
    \begin{align}\label{coverball}
     B^x \subseteq \bigcup\limits_k Q_k^x.
\end{align}
Using \eqref{sublinearofHbeta} and \eqref{coverball}, we obtain 
\begin{align*}
    \sum_k \int_{B^x \cap Q_k^x} |f| \; d\mathcal{H}^\beta_\infty\geq  \frac{1}{C_\beta^2}\int_{B^x} |f| \; d \mathcal{H}^\beta_\infty > \frac{t \omega_\beta r_x^\beta}{C_\beta^2} ,
\end{align*}
and therefore there exists $Q^x:= Q_k^x$ for some $k$ such that
\begin{align*}
    \int_{B^x \cap Q^x} |f| \; d\mathcal{H}^\beta_\infty \geq \frac{t \omega_\beta r_x^\beta}{C_\beta^2 2^n} .
\end{align*}
Since \eqref{equivalentsize} implies that $ r_x^\beta > \left( \frac{\ell(\Tilde{Q_x})}{4} \right)^\beta$, we obtain
\begin{align*}
    \int_{ Q^x} |f| \; d \mathcal{H}^\beta_\infty \geq \int_{B^x \cap  Q^x} |f| \; d\mathcal{H}^\beta_\infty \geq \frac{t \omega_\beta r_x^\beta}{C_\beta^2 2^n} > \frac{t \omega_\beta}{C_\beta^2 2^{n+2\beta}} \ell( Q^x)^\beta.
\end{align*}
In particular,
\begin{align*}
    \frac{1}{\ell( Q^x)^\beta} \int_{ Q^x} |f| \;d\mathcal{H}^{\beta,Q_0}_\infty > c t
\end{align*}
for 
\begin{align*}
c:=\frac{ \omega_\beta}{C_\beta^3 2^{n+2\beta}}.
\end{align*}
It follows that 
\begin{align*}
    Q^x \subseteq L_{ct},
\end{align*}
while \eqref{equivalentsize} and \eqref{nondisjoint} imply that $B^x \subseteq 3 Q^x$, which proves the claim.

It follows from the previous claim that
\begin{align*}
    \bigcup\limits_{x \in E_t} B^x \subseteq  \bigcup_{x \in E_t} 3Q^x \subseteq \{ 3Q: Q \subseteq \mathcal{D}(Q_0) \text{ and } Q \subseteq  L_{ct} \}.  
\end{align*}
The subadditivity of $\mathcal{H}^\beta_\infty $ then implies that
\begin{align*}
    \mathcal{H}^\beta_\infty (E_t) &\leq   \mathcal{H}^\beta_\infty \left(  \bigcup_{x\in E_t} B(x,r_x) \right)\\
    &\leq   \mathcal{H}^\beta_\infty \left(  \bigcup_{x \in E_t} 3 Q^x \right)    \\
    &\leq \mathcal{H}^\beta_\infty\left( \{ 3Q: Q \subseteq \mathcal{D}(Q_0) \text{ and } Q \subseteq  L_{c t} \} \right).
   \end{align*}
 For $Q \subseteq \mathcal{D}(Q_0) \text{ and } Q \subseteq  L_{ct}$ write
 \begin{align*}
 3Q = \bigcup_{i=1}^{3^n} \tau_i Q,
 \end{align*}
 for appropriate translation operators $\{\tau_i\}_{i=1}^{3^n}$.  Then subadditivity and translation invariance of $ \mathcal{H}^\beta_\infty $ implies
 \begin{align*}
 \mathcal{H}^\beta_\infty\left( \{ 3Q: Q \subseteq \mathcal{D}(Q_0) \text{ and } Q \subseteq  L_{c t} \} \right) &= \mathcal{H}^\beta_\infty\left( \{ \cup_{i=1}^{3^n} \tau_i Q: Q \subseteq \mathcal{D}(Q_0) \text{ and } Q \subseteq  L_{c t} \} \right) \\
 &\leq \sum_{i=1}^{3^n} \mathcal{H}^\beta_\infty\left( \{ \tau_i Q: Q \subseteq \mathcal{D}(Q_0) \text{ and } Q \subseteq  L_{c t} \} \right) \\
 &= 3^n  \mathcal{H}^\beta_\infty\left( \{Q: Q \subseteq \mathcal{D}(Q_0) \text{ and } Q \subseteq  L_{c t} \} \right).
 \end{align*}
The combination of the two preceding chains of inequalities yields 
\begin{align*}
 \mathcal{H}^\beta_\infty (E_t) &\leq  3^n \mathcal{H}^{\beta}_\infty\left( \{ Q: Q \subseteq \mathcal{D}(Q_0) \text{ and } Q \subseteq  L_{ct} \} \right) \nonumber\\
    &=  3^n \mathcal{H}^{\beta}_\infty\left(  L_{ct} \right), \nonumber
\end{align*}
where the last equality follows from the fact that  
\begin{align*}
   \{ Q: Q \subseteq \mathcal{D}(Q_0) \text{ and } Q \subseteq  L_{ct} \}  = L_{ct}.
\end{align*}
We thus obtain the claimed result with
\begin{align*}
     C:=3^n.
\end{align*}
\end{proof}

The next lemma is the simplest case of the Marcinkiewicz interpolation theorem for capacities.  One can argue this result by the computation of interpolation spaces for capacitary integration spaces in \cite{Cerda, CerdaCollMartin,CerdaMartinSilvestre} and an application of an abstract interpolation result for quasi-linear operators on Banach spaces, though we include a direct proof for completeness and the convenience of the reader.
\begin{lemma}\label{interpolationofcapacity}
    Let $C$ be a capacity, i.e.
\begin{enumerate}
\item  $C(\emptyset)=0$;
\item  If $E \subset F$, then $C(E)\leq C(F)$;
\item If $E \subset \cup_{i=1}^\infty E_i$ then 
\begin{align*}
C(E) &\leq \sum_{i=1}^\infty C(E_i).
\end{align*}
\end{enumerate}
Suppose T is a quasi-linear operator: \begin{align}\label{quasilinearofT}
        |T(f_1 + f_2)(x)| \leq K \left(|T(f_1)(x)| + |T(f_2)(x)| \right)
    \end{align}
    for some constant $K$ (see e.g. \cite[1.3.4 on p.~33]{Grafakos}). If there exist $A_1$, $A_2 >0$ such that 
\begin{align}\label{weaktype(11)ofC}
    C\left( \{ x\in \mathbb{R}^n : \left|T f(x) \right|  >t\} \right) \leq \frac{A_1}{t} \Vert f \Vert_{L^1(C)}
\end{align}
for all $C$-quasieverywhere defined functions $f \in L^1(C)$ and 
\begin{align}\label{strontypeinfinityofC}
    \Vert Tf \Vert_{L^\infty(C)} \leq A_2 \Vert f \Vert_{L^\infty(C)}
\end{align}
for all $C$-quasieverywhere defined functions  $f \in L^\infty(C)$, then for $1<p<\infty$ we have the estimate
\begin{align*}
    \Vert Tf \Vert_{L^p(C)} \leq A_3 \Vert f \Vert_{L^p(C)},
\end{align*}
    where $A_3$ depending only on $K$, $A_1$, $A_2$ and $p$.
\end{lemma}
\begin{proof}
    Fix $f \in L^p(C)$ and we split $f = f^\lambda + f_\lambda$, where $f^\lambda = f \chi_{ \{|f|> \frac{\lambda}{2A_2 K} \}}$ and $f_\lambda = f \chi_{ \{|f| \leq  \frac{\lambda}{2 A_2 K} \}}$. Using (\ref{quasilinearofT}) and (\ref{strontypeinfinityofC}), we obtain
\begin{align*}
   T f (x) &  \leq K \left( |T f^\lambda (x)| +|T f_\lambda (x)| \right)\\
    &\leq K \left( |T f^\lambda (x)| +A_2 \frac{\lambda}{2 A_2 K} \right)\\
    &= K |T f^\lambda (x)|+ \frac{\lambda}{2},
\end{align*}
and thus
\begin{align*}
    C \left( \{ x \in \mathbb{R}^n :|T f(x)| >\lambda \} \right) \leq   C \left(   \{ x\in \mathbb{R}^n: |T f^\lambda (x)| > \frac{\lambda}{2K} \} \right).
\end{align*}
Moreover, using inequality \eqref{weaktype(11)ofC}, we deduce that 
\begin{align*}
 C \left( \{ x \in \mathbb{R}^n :| T f(x)| >\lambda \} \right) 
 &\leq C \left( \{ x\in \mathbb{R}^n :|T f^\lambda(x) |> \frac{\lambda}{2K} \}\right)\\
& \leq \frac{2 A_1 K}{\lambda} \int_{\mathbb{R}^n} |f^\lambda | \; dC.
\end{align*}
We note that 
\begin{align*}
    \int_{\mathbb{R}^n} |f^\lambda | \; dC &=  \int^{\frac{\lambda}{2 A_2 K}}_0 C\left( \{ x \in \mathbb{R}^n: |f^\lambda (x)| > t \} \right) \;dt +  \int^\infty_{\frac{\lambda}{2 A_2 K}} C\left( \{ x \in \mathbb{R}^n: |f^\lambda (x)| > t \} \right) \;dt\\
    &= \frac{\lambda}{2 A_2 K}  C\left( \{ x \in \mathbb{R}^n: |f (x)| > \frac{\lambda}{2 A_2 K} \} \right)  +  \int^\infty_{\frac{\lambda}{2 A_2 K}} C\left( \{ x \in \mathbb{R}^n: |f (x)| > t \} \right) \;dt,
\end{align*}
and thus 
\begin{align}\label{midestimateofmaximalcapacityversion}
    & C \left( \{ x \in \mathbb{R}^n :|T f(x)| >\lambda \} \right) \\
   \leq  & \frac{2A_1 K}{\lambda} \left( \frac{\lambda}{2A_2 K}C \left( \{ x \in \mathbb{R}^n: |f(x)| > \frac{\lambda}{2 A_2 K} \} \right) 
    + \int^\infty_{\frac{\lambda}{2 A_2 K}} C\left( \{ x \in \mathbb{R}^n: |f (x)| > t \} \right) \;dt \right) .  \nonumber
\end{align}
Using the formula $ \int_{\mathbb{R}^n} \left|T f \right|^p \; d C   = p \int^\infty_0 \lambda^{p-1} C \left(\{ x\in \mathbb{R}^n:|T f(x)| > \lambda\} \right) \; d\lambda $ and the inequality \eqref{midestimateofmaximalcapacityversion}, we deduce that \begin{align} \label{combinest0capacity}
    \int_{\mathbb{R}^n} \left| T f \right|^p \; d C  & \leq  \frac{A_1 p}{A_2}   \int^\infty_0 \lambda^{p-1} C \left( \{ x \in \mathbb{R}^n: |f(x)| > \frac{\lambda}{2 A_2 K} \} \right)  \; d\lambda\\
    & +   2A_1 K p \int^\infty_0 \lambda^{p-2} \int^\infty_{\frac{\lambda}{2 A_2 K}} C\left( \{ x \in \mathbb{R}^n: |f (x)| > t \} \right) \;dt \;d\lambda. \nonumber
\end{align}
The change of variables  $s=\lambda/(2A_2K)$ yields
\begin{align}\label{combinest1capacity}
  p\int_0^\infty \lambda^{p-1} C \left( \{ x \in \mathbb{R}^n: |f(x)| > \frac{\lambda}{2 A_2 K} \} \right) \; d\lambda =(2A_2 K)^p  \int_{\mathbb{R}^n}  |f|^p \; d C
 \end{align}
while from Tonelli’s theorem and basic calculus we obtain
\begin{align}\label{combinest2capacity}
&  p  \int^\infty_0 \lambda^{p-2} \int^\infty_{\frac{\lambda}{2 A_2 K}} C\left( \{ x \in \mathbb{R}^n: |f (x)| > t \} \right) \;dt \;d\lambda\\
=& p \int^\infty_0 \int^{ 2 A_2 K t}_0 \lambda^{p-2} C\left( \{ x \in \mathbb{R}^n: |f (x)| > t \} \right) \; d \lambda \;dt \nonumber \\
 =& \frac{(2A_2K)^{p-1}}{p-1} \int_{\mathbb{R}^n}  |f|^p \; d C.\nonumber
\end{align}
Therefore, finally the combination of \eqref{combinest0capacity}, \eqref{combinest1capacity} and \eqref{combinest2capacity} yields 
\begin{align*}
 \int_{\mathbb{R}^n} \left(T f \right)^p \; d C   \leq  A_3 \int_{\mathbb{R}^n}  |f|^p \; d C,
\end{align*}
for
\begin{align*}
A_3:= \left(\frac{A_1}{A_2}(2A_2 K)^p + 2A_1K \frac{(2A_2K)^{p-1}}{p-1}\right)
\end{align*}
and the proof is complete.
\end{proof}

We next collect some useful properties regarding the Riesz capacities, we refer the readers to \cite{AH} for further details.

\begin{lemma}\label{Comparisonbetweencapmaximal}
    Let $0<\gamma <\alpha <\frac{n}{s}$, $1<s<\frac{n}{\alpha}$, $f$ be a real-valued function. Then there exists a constant $C = C(n, \alpha, \gamma, s)>0$ such that
\begin{align*}
    \mathcal{M}_{{\rm cap}_{\gamma,s}}f(x) \leq C\left(\mathcal{M}_{{\rm cap}_{\alpha,s}}f(x)\right)^{\frac{n-\gamma s}{n-\alpha s}}
\end{align*}
\end{lemma}

\begin{proof}

We begin with an inequality between capacities of different exponents which can be found in \cite[Theorem 5.5.1 on p.~148]{AH}: There exists a constant $C_1= C_1(n,\alpha, s ,\gamma)>0$ such that
\begin{align}\label{cap_set_inequality}
{\rm cap}_{\gamma,s}(E)^{\frac{1}{n-\gamma s}}\leq C_1 {\rm cap}_{\alpha,s}(E)^{\frac{1}{n-\alpha s}},
\end{align}
which is valid for $0<\gamma <\alpha <\frac{n}{s}$ and $1<s<\frac{n}{\alpha}$. The observation that ${\rm cap}_{0,s}(E) \equiv c|E|$ for some constant $c$ and the usual Sobolev relation between the Lebesgue measure and capacity (which this inequality is an extension of) allow us to extend the range of exponents to $0\leq \gamma <\alpha <\frac{n}{s}$.

By \eqref{cap_set_inequality} and a change of variables we have
\begin{align*}
\int_{B(x,r)}|f| \;d{\rm cap}_{\gamma,s}&\leq C_{1}^{\frac{n-\gamma s}{n-\alpha s}}\int_{0}^{\infty}{\rm cap}_{\alpha,s}(\{  y \in \mathbb{R}^n:|f|\chi_{B(x,r)}>\lambda\})^{\frac{n-\gamma s}{n-\alpha s}} \;d\lambda\\
&=\frac{n-\gamma s}{n-\alpha s} C_{1}^{\frac{n-\gamma s}{n-\alpha s}}\int_{0}^{\infty}{\rm cap}_{\alpha,s}(\{  y \in \mathbb{R}^n:|f|^{\frac{n-\alpha s}{n-\gamma s}}\chi_{B(x,r)}> \lambda\})^{\frac{n-\gamma s}{n-\alpha s}} \lambda^{\frac{n-\gamma s}{n-\alpha s}-1}\; d\lambda.
\end{align*}
The validity of Chebychev's inequality for the Choquet integral implies the estimate
\begin{align*}
   {\rm cap}_{\alpha,s}(\{y\in \mathbb{R}^n: |f|^{\frac{n-\alpha s}{n-\gamma s}}\chi_{B(x,r)}>\lambda\})^{\frac{n-\gamma s}{n-\alpha s}-1} \leq \frac{1}{\lambda^{\frac{n-\gamma s}{n-\alpha s}-1}} \left( \int_{B(x,r)}|f|^{\frac{n-\alpha s}{n-\gamma s}} \;d{\rm cap}_{\alpha,s}\right)^{\frac{n-\gamma s}{n-\alpha s}-1},
\end{align*}
and therefore
\begin{align*}
\int_{B(x,r)}&|f| \;d{\rm cap}_{\gamma,s} \\
&\leq \frac{n-\gamma s}{n-\alpha s} C_{1}^{\frac{n-\gamma s}{n-\alpha s}}\left(\int_{B(x,r)}|f|^{\frac{n-\alpha s}{n-\gamma s}} \;d{\rm cap}_{\alpha,s}\right)^{\frac{n-\gamma t}{n-\alpha s}-1}\int_{0}^{\infty}{\rm cap}_{\alpha,s}(\{ y \in \mathbb{R}^n:|f|^{\frac{n-\alpha s}{n-\gamma s}}\chi_{B(x,r)}>\lambda\})\;d\lambda\\
&=\frac{n-\gamma s}{n-\alpha s} C_{1}^{\frac{n-\gamma s}{n-\alpha s}}\left(\int_{B(x,r)}|f|^{\frac{n-\alpha s}{n-\gamma s}} \;d{\rm cap}_{\alpha,s}\right)^{\frac{n-\gamma s}{n-\alpha s}}.
\end{align*}

As a consequence, we for any $r>0$ we have 
\begin{align*}
\frac{1}{{\rm cap}_{\gamma,s}(B(0,1))r^{n-\gamma s}}\int_{B(x,r)}|f| \;d{\rm cap}_{\gamma,s}\leq C \left(\mathcal{M}_{{\rm cap}_{\alpha,s}}|f|^{\frac{n- \alpha s}{n-\gamma s}}(x)\right)^{\frac{n-\gamma s}{n-\alpha s}}.
\end{align*}
where 
\begin{align*}
C:=\frac{n-\gamma s}{n-\alpha s} C_{1}^{\frac{n-\gamma s}{n-\alpha s}} \frac{{\rm cap}_{\alpha,s}(B(0,1))^\frac{n-\alpha s}{n-\gamma s}}{{\rm cap}_{\gamma,s}(B(0,1))}.
\end{align*}
The conclusion follows by taking the supremum in $r$.
\end{proof}

\section{Proofs of the Main Results}
We now prove Theorem \ref{maintheorem1}.
\begin{proof}
   We first prove $(ii)$.  Let $ Q_0 = [0,1]^n  \subseteq \mathbb{R}^n$.  Then Lemma \ref{equivoftwomaximal} and inequality \eqref{dyadicweaktypeestimate11} imply that there exists a constant $A_1>0$ such that
\begin{align*}
 \mathcal{H}^\beta_\infty \left( \{ x \in \mathbb{R}^n : \mathcal{M}_{\mathcal{H}^\beta_\infty} f(x) >\lambda \} \right) 
&\leq C  \mathcal{H}^\beta_\infty \left(\{x \in\mathbb{R}^n : \mathcal{M}_\infty^{\beta,Q_0} f(x) >  c\lambda \} \right)\\
& \leq \frac{A_1}{\lambda} \int_{\mathbb{R}^n} |f | \; d\mathcal{H}^{\beta }_\infty.
\end{align*}
This completes the demonstration of $(ii)$.  To prove $(i)$, we note that the operator $T = \mathcal{M}_{\mathcal{H}^\beta_\infty}$ is quasisublinear, which follows from the computation
\begin{align*}
      \frac{1}{\omega_\beta r^\beta}\int_{B(x,r)} |f_1 +f_2| \; d\mathcal{H}^\beta_\infty \leq2 \left(  \frac{1}{\omega_\beta r^\beta} \int_{B(x,r)} |f_1 |  \; d\mathcal{H}^\beta_\infty+  \frac{1}{\omega_\beta r^\beta} \int_{B(x,r)}  |f_2 | \; d\mathcal{H}^\beta_\infty \right),
  \end{align*}
so that
\begin{align*}
    \mathcal{M}_{\mathcal{H}^\beta_\infty} (f_1 + f_2) (x) &  \leq 2 \left( \mathcal{M}_{\mathcal{H}^\beta_\infty} f_1 (x) +\mathcal{M}_{\mathcal{H}^\beta_\infty} f_2 (x) \right).
\end{align*}  

We claim that
\begin{align*}
    \Vert \mathcal{M}_{\mathcal{H}^\beta_\infty} f \Vert_{L^\infty (\mathcal{H}^\beta_\infty)} \leq \Vert  f \Vert_{L^\infty(\mathcal{H}^\beta_\infty)}.
\end{align*}
If so, an application of Lemma \ref{interpolationofcapacity} with $T = \mathcal{M}^\beta_\infty$ and $C= \mathcal{H}^\beta_\infty$ yields 
\begin{align*}
    \Vert \mathcal{M}_{\mathcal{H}^\beta_\infty} f \Vert_{L^p (\mathcal{H}^\beta_\infty)} \leq A_3 \Vert f \Vert_{L^p(\mathcal{H}^\beta_\infty)},
\end{align*}
which is the desired result.  To this end,  we observe that
\begin{align*}
\frac{1}{\mathcal{H}^\beta_\infty(B(x,r))} \int_{B(x,r)} |f|\;d\mathcal{H}^\beta_\infty &=\frac{1}{\mathcal{H}^\beta_\infty(B(x,r))} \int_0^\infty  \mathcal{H}^\beta_\infty(\{y \in B(x,r): |f(y)|>t\})\;dt \\
&=\frac{1}{\mathcal{H}^\beta_\infty(B(x,r))} \int_0^{\|f\|_{L^\infty(\mathcal{H}^\beta_\infty)}}  \mathcal{H}^\beta_\infty(\{y \in B(x,r): |f(y)|>t\})\;dt \\
&\leq\|f\|_{L^\infty(\mathcal{H}^\beta_\infty)},
\end{align*}
as
\begin{align*}
\mathcal{H}^\beta_\infty(\{y \in B(x,r): |f(y)|>t\}) = 0
\end{align*}
for $t>\|f\|_{L^\infty(\mathcal{H}^\beta_\infty)}$, since
\begin{align*}
\|f\|_{L^\infty(\mathcal{H}^\beta_\infty)} := \inf\{ \lambda>0 : \mathcal{H}^\beta_\infty(\{y : |f(y)|>\lambda\}) = 0\}.
\end{align*}
Taking the supremum over radii and in $x$ yields the claim and the proof is complete.
\end{proof}

We next commence with the 
\begin{proof}[Proof of Corollary \ref{corollary_LDT}]
Let $f \in L^p(\mathcal{H}^\beta_\infty)$.  For any $g \in C_c(\mathbb{R}^n)$, one has
\begin{align*}
\lim_{r \to 0} \frac{1}{\omega_\beta r^\beta}\int_{B(x,r)}|g-g^*(x)|^p \;d\mathcal{H}^\beta_\infty &= 0,\\
\frac{1}{\omega_\beta r^\beta}\int_{B(x,r)}|g^*(x)-f^*(x)|^p \;d\mathcal{H}^\beta_\infty&=|g^*(x)-f^*(x)|^p,\\
\limsup_{r \to 0} \frac{1}{\omega_\beta r^\beta}\int_{B(x,r)}|f-g|^p \;d\mathcal{H}^\beta_\infty  &\leq \mathcal{M}_{\mathcal{H}^\beta_\infty} |f-g|^p
\end{align*}
and therefore by quasisubadditivity of $t \mapsto t^p$ and of the Choquet integral we have
\begin{align*}
\limsup_{r\to 0}\frac{1}{\omega_\beta r^\beta}\int_{B(x,r)}&|f-f^*(x)|^p \;d\mathcal{H}^\beta_\infty \\
&\leq C\left(\mathcal{M}_{\mathcal{H}^\beta_\infty} |f-g|^p +|g^*(x)-f^*(x)|^p\right).
\end{align*}
Thus for any $\lambda>0$, using subadditivity of the capacity, the weak-type estimate obtained in Theorem \ref{maintheorem1}, and Chebychev's inequality, we find
\begin{align*}
\mathcal{H}^\beta_\infty\left( \left\{ \limsup_{r\to 0}\frac{1}{\omega_\beta r^\beta}\int_{B(x,r)}|f-f^*(x)|^p>\lambda\right\} \right) \leq \frac{C'}{\lambda^p} \|f-g\|_{L^p(\mathcal{H}^\beta_\infty)}^p.
\end{align*}
By density of $C_c(\mathbb{R}^n)$ in $L^p(\mathcal{H}^\beta_\infty)$ we conclude that for any $\lambda>0$
\begin{align*}
    \mathcal{H}^\beta_\infty\left( \left\{ \limsup_{r\to 0}\frac{1}{\omega_\beta r^\beta}\int_{B(x,r)}|f-f^*(x)|^p>\lambda\right\} \right)=0.
\end{align*}
But then writing
\begin{align*}
    &\left\{ \limsup_{r\to 0}\frac{1}{\omega_\beta r^\beta}\int_{B(x,r)}|f-f^*(x)|^p>0\right\} \\
    &= \bigcup_{n \in \mathbb{N}} \left\{ \limsup_{r\to 0}\frac{1}{\omega_\beta r^\beta}\int_{B(x,r)}|f-f^*(x)|^p>\frac{1}{n}\right\},
\end{align*}
by countable subadditivity of the capacity we find that
\begin{align*}
  \mathcal{H}^\beta_\infty\left( \left\{ \limsup_{r\to 0}\frac{1}{\omega_\beta r^\beta}\int_{B(x,r)}|f-f^*(x)|^p>0\right\} \right)=0,
\end{align*}
which along with the fact that $f^*$ is $\mathcal{H}^\beta_\infty$ quasieverywhere equal to the limit of the averages completes the proof.
\end{proof}

   We now prove Corollary \ref{maintheorem2}.
\begin{proof}
Without loss of generality, we may assume that $f \geq 0$. \\

(i) Let $p>\frac{\gamma}{\beta}$. Applying Lemma \ref{maximaldimetionalchange} and Theorem \ref{maintheorem1} $(i)$, we deduce that there exists a constant $C = C(\gamma, \beta, p, n)$ such that 
\begin{align*}
    \int_{\mathbb{R}^n} \left( \mathcal{M}_{\mathcal{H}^\beta_\infty} f \right)^p \; d \mathcal{H}^\gamma_\infty &\leq \left( \frac{\beta}{\gamma}\right)^p \int_{\mathbb{R}^n} \left( \mathcal{M}_{\mathcal{H}^\gamma_\infty} f ^{\frac{\gamma}{\beta}}  \right)^{ \frac{\beta}{\gamma}p} \; d \mathcal{H}^\gamma_\infty\\
    &\leq C \int_{\mathbb{R}^n} f^p \; d\mathcal{H}^\gamma_\infty,
\end{align*}
which completes the proof.

(ii) Let $p = \frac{\gamma}{\beta}$. An application of Lemma \ref{maximaldimetionalchange} yields
\begin{align*}
    \mathcal{H}^\gamma_\infty \left( \{ x\in \mathbb{R}^n: \mathcal{M}_{\mathcal{H}^\beta_\infty} f(x) >t  \}   \right)&\leq   \mathcal{H}^\gamma_\infty \left( \left\{ x\in \mathbb{R}^n: \mathcal{M}_{\mathcal{H}^\gamma_\infty} |f|^{\frac{\beta}{\gamma}} (x) >  \left( \frac{\gamma}{\beta} t\right)^{\frac{\beta}{\gamma}} \right\}   \right) \\
    &=\mathcal{H}^\gamma_\infty \left( \left\{ x\in \mathbb{R}^n: \mathcal{M}_{\mathcal{H}^\gamma_\infty} |f|^{p} (x) >  \left( \frac{\gamma}{\beta} t\right)^{p} \right\}   \right) \nonumber
\end{align*}
for $p=\frac{\beta}{\gamma}$.  The desired inequality follows from Theorem \ref{maintheorem1} $(ii)$ applied to the function $|f|^p$.
\end{proof}

\begin{lemma}\label{fill_in_the_gaps}
    Let $0 < \alpha < n$, $1 <s <\infty$ and $\Gamma_{\alpha,s}$ be defined as 
    \begin{align*}
        \Gamma_{\alpha,s} (f) := \inf \{\Vert \varphi \Vert_{L^s(\mathbb{R}^n)}^s : I_\alpha \ast \varphi (x) \geq |f(x)| \text{ for all }  x \in \mathbb{R}^n\}
    \end{align*}
    for every $f: \mathbb{R}^n \to \mathbb{R}$. Then there exists a constant $C=C(s)>0$ such that
    \begin{align}\label{equiv_norm}
        \Gamma_{\alpha, s} (f) \leq C \int_{\mathbb{R}^n} |f |^s \; d {\rm cap}_{\alpha,s}.
    \end{align}
     
\end{lemma}
\begin{proof}
    It is clear that 
\begin{align}\label{subadditive1}
    \Gamma_{\alpha,s} (f_1) \leq \Gamma_{\alpha,s} (f_2) \text{ if } |f_1| \leq |f_2| 
\end{align}
    and  
    \begin{align}\label{subaddtive2}
        \Gamma_{\alpha,s} (\chi_E) =  {\rm cap}_{\alpha,s} (E) \text{ for all } E \subseteq \mathbb{R}^n.
    \end{align}
Let $f:\mathbb{R}^n \to \mathbb{R}$ be given and without loss of generality we assume that the right hand side of \eqref{equiv_norm} is finite.  Define
\begin{align*}
K_j &:= \{ x\in \mathbb{R}^n: 2^j \leq |f(x)| \leq 2^{j+1}\},\\
E_j &:= \{ x\in \mathbb{R}^n: 2^j \leq |f(x)|\},
    \end{align*}
    and note that
    \begin{align}\label{subaddtivecon2}
    \int_{\mathbb{R}^n} |f|^s \; d  {\rm cap}_{\alpha,s} & =  \sum^\infty_{j=-\infty}s \int^{2^{j}}_{2^{j-1}} t^{s-1 }  {\rm cap}_{\alpha,s} (\{x \in \mathbb{R}^n: |f(x)|>t \}) \;dt\\ 
&\geq   \sum^\infty_{j=-\infty} 2^{js} (1-2^{-s})   {\rm cap}_{\alpha,s}  (\{x \in \mathbb{R}^n: \nonumber|f(x)|\geq 2^{j} \}) \\\nonumber
& = (1-2^{-s})  \sum^\infty_{j=-\infty} 2^{sj}  \;{\rm cap}_{\alpha,s} (E_{j}). \nonumber
\end{align}

The preceding inequality and the chain of inequalities
\begin{align*}
\Gamma_{\alpha, s} (f \chi_{K_j}) \leq 2^{j+1} \Gamma_{\alpha, s} (\chi_{K_j}) = 2^{j+1} {\rm cap}_{\alpha,s} (\chi_{K_j}) \leq 2^{j+1} {\rm cap}_{\alpha,s} (\chi_{E_j})
\end{align*}
yields that $\Gamma_{\alpha, s} (f \chi_{K_j})$ is finite for each $j$.  Therefore, for any $\varepsilon >0$, the definition of $\Gamma_{\alpha,s}$ implies that for each $j$ we can find  $\varphi_j$ such that 
\begin{align*}
    I_\alpha \ast\varphi_j \geq |f \chi_{K_j}|
\end{align*}
and 
\begin{align*}
    \Vert \varphi_j \Vert_{L^s (\mathbb{R}^n)}^s \leq \Gamma_{\alpha, s} (f \chi_{K_j}) + \frac{\varepsilon}{2^j}.
\end{align*}
Set $\varphi : = \sup_j \varphi_j$, so that
\begin{align*}
    I_\alpha \ast \varphi \geq f \chi_{\cup_j K_j}
\end{align*}
and 
\begin{align*}
    \Vert \varphi \Vert_{L^s(\mathbb{R}^n)}^s \leq \sum_j \Gamma_{\alpha,s} (f \;\chi_{\cup_j K_j}) + \varepsilon.
\end{align*}
It follows that 
\begin{align}\label{subaddtive3}
    \Gamma_{\alpha,s} (f  \chi_{\cup_j K_j}) \leq \sum_j \Gamma_{\alpha,s } (f \chi_{K_j}).
\end{align}
 The combination of \eqref{subadditive1}, \eqref{subaddtive2}, and \eqref{subaddtive3} yields
    \begin{align}\label{subaddtivecon1}
        \Gamma_{\alpha,s}(f) & \leq \sum_{j= - \infty}^\infty \Gamma_{\alpha,s} ( f \chi_{K_j})\\\nonumber
        &\leq \sum_{j= - \infty}^\infty 2^{(j+1)s } \Gamma_{\alpha,s} (\chi_{K_j})\\& \leq \sum_{j= - \infty}^\infty 2^{(j+1)s }  {\rm cap}_{\alpha,s} (E_j).\nonumber
    \end{align}
Finally, the inequalities \eqref{subaddtivecon2} and
 \eqref{subaddtivecon1} together imply
\begin{align*}
    \Gamma_{\alpha,s} (f) \leq C \int_{\mathbb{R}^n} |f|^s \; d  {\rm cap}_{\alpha,s}
\end{align*}
for 
\begin{align*}
C:= \frac{2^s}{1-2^{-s}}.
\end{align*}
The proof is complete.
\end{proof}

The proof of Theorem \ref{main 1} relies on the strong capacitary inequality and a weak-type estimate we establish in the following lemma, which is in the spirit of \cite[Lemma 6.2.2 on p.~159]{AH}.

\begin{lemma}\label{main pro}
Let $0<\alpha<n$, $1<s<n/\alpha$, and $\varphi\geq 0$ be a measurable function on $\mathbb{R}^{n}$. Denote by  
\begin{align*}
E_{t}=\{x\in\mathbb{R}^{n}:{\mathcal M}_{{\rm cap}_{\alpha,s}}\left((I_{\alpha}\ast\varphi)^{s}\right)(x)>t\},\quad t>0.
\end{align*}
Then there exists a constant depending only on $C = C(n,\alpha,s)>0$ such that
\begin{align*}
{\rm cap}_{\alpha,s}(E_{t})\leq \frac{C}{t}\|\varphi\|_{L^{s}(\mathbb{R}^{n})}^{s}
\end{align*}
holds for all $t>0$.
\end{lemma}

\begin{proof}
We first observe that for every $x \in E_t$, there exists an open ball, denoted as $B(x,r_x)$, satisfying 
\begin{align*} 
\frac{1}{c_{\alpha,s} r_x^{n-\alpha s}}\int_{B(x,r_x)}(I_{\alpha}\ast\varphi)^{s} \; d{\rm cap}_{\alpha,s}>t,
\end{align*}
where for convenience of display we use the notation
\begin{align*}
c_{\alpha,s}:= {\rm cap}_{\alpha,s}(B(0,1)).
\end{align*}
It follows that 
\begin{align}\label{621.5}
c_{\alpha,s} r_x^{n-\alpha s}t&<\int_{B(x,r_x)}(I_{\alpha}\ast\varphi)^{s} \; d{\rm cap}_{\alpha,s}\\ \nonumber
&\leq\int_{B(x,r_x)}2^{s-1}\left((I_{\alpha}\ast\phi_x)^{s}+(I_{\alpha}\ast\psi_x)^{s}\right) \; d{\rm cap}_{\alpha,s}\\ \nonumber
&\leq 2^{s}\left(\int_{B(x,r_x)}(I_{\alpha}\ast\phi_x)^{s} \; d{\rm cap}_{\alpha,s}+\int_{B(x,r_x)}(I_{\alpha}\ast\psi_x)^{s} \; d{\rm cap}_{\alpha,s}\right), \nonumber
\end{align}
where
\begin{align*}
  \phi_x : = \varphi \chi_{B(x,  4 r_x)}
\end{align*}
and 
\begin{align*}
    \psi_x : = \varphi  - \phi_x.
\end{align*}
Therefore, if we define 
\begin{align}\label{622}
U_t:= \left\{ x\in E_t: 
\frac{1}{c_{\alpha,s} r_x^{n-\alpha s}}\int_{B(x,r_x)}(I_{\alpha}\ast\phi_x)^{s} \;d{\rm cap}_{\alpha,s}>\frac{t}{2^{s+1}}\right\}
\end{align}
and
\begin{align}\label{623}
V_t:= \left\{ x\in E_t:
\frac{1}{c_{\alpha,s} r_x^{n-\alpha s}}\int_{B(x,r_x)}(I_{\alpha}\ast\psi_x)^{s} \;d{\rm cap}_{\alpha,s}>\frac{t}{2^{s+1}} \right\} ,
\end{align}
we obtain
\begin{align*}
    E_t \subseteq U_t \cup V_t
\end{align*}
by inequality (\ref{621.5}).

We now claim that there exists a constant $C' = C'(n, \alpha, s)>0$ such that 
\begin{align}\label{621.6}
   {\rm cap}_{\alpha,s} (U_t)\leq \frac{C'}{t} \Vert \varphi \Vert_{L^s(\mathbb{R}^n)}^s.
\end{align}
To this end, we apply Theorem 7.1.1 from \cite{AH} which states that there exists $C_1 =C_1(n,\alpha,s)>0$ such that 
\begin{align}\label{CSI}
\int_{0}^{\infty}{\rm cap}_{\alpha,s}\left(\left\{x\in\mathbb{R}^{n}:(I_{\alpha}\ast\varphi(x))^{s}>t\right\}\right) \; dt\leq C_{1}\|\varphi\|_{L^{s}(\mathbb{R}^{n})}^{s},
\end{align}
and deduce that for every $x \in U_t$, we have
\begin{align*}
c_{\alpha,s} r_x^{n-\alpha s} \frac{t}{2^{s+1}}  &< \int_{\mathbb{R}^{n}}(I_{\alpha}\ast\phi_x)^{s} \; d{\rm cap}_{\alpha,s} \\
&\leq C_{1}\int_{\mathbb{R}^{n}}\phi_x(y)^{s} \;dy\\
&=C_{1}\int_{B(x, 4r_x)}\varphi(y)^{s} \;dy
\end{align*}
That is, for every $x\in U_t$, there exists $r_x >0$ such that
\begin{align}\label{625}
C_2 \int_{B(x, 4r_x)}\varphi(y)^{s} \;dy >  r_x^{n-\alpha s}t
\end{align}
for
\begin{align*}
    C_2:=\frac{C_1 2^{s+1}}{c_{\alpha,s}}.
\end{align*}
An application of the Vitali covering lemma \cite[Theorem 1 on p.~27]{EvansGariepy} to the collection of balls $\mathcal{F}: = \{ B(x,  4 r_x)\}_{x \in U_t}$, we obtain a countable subcollection $\mathcal{G}$ of disjoint balls in $\mathcal{F}$ such that
\begin{align*}
    \bigcup_{ B \in \mathcal{F}} B  \subseteq  \bigcup_{ B \in \mathcal{G}}  \hat{B},
\end{align*}
where $\hat{B}$ denotes the open ball with same centre with $B$ but 5 times the radius of $B$. As a result, we obtain the estimate
\begin{align*}
    {\rm cap}_{\alpha,s} (U_t) &\leq  {\rm cap}_{\alpha,s}\left( \bigcup_{B (x, 4  r_x) \in \mathcal{F}} B(x, 4 r_x)  \right)\\
    & \leq {\rm cap}_{\alpha,s} \left( \bigcup_{B (x, 4 r_x)  \in \mathcal{G}} B(x, 20 r_x)  \right)\\
    &\leq \sum_{B (x, 4 r_x) \in \mathcal{G}}{\rm cap}_{\alpha,s} ( B(x, 20 r_x)  )\\
    &= 20^{n-\alpha s} c_{\alpha,s}  \sum_{B (x, 4 r_x)  \in \mathcal{G}} r_x^{n-\alpha s }.
\end{align*}
The preceding inequality in combination with \eqref{625} yields the estimate
\begin{align*}
    {\rm cap}_{\alpha,s} (U_t) &\leq \frac{C'}{t} \sum_{B (x, 4 r_x)  \in \mathcal{G}} \int_B \varphi (y)^s  \; dy\\
    &\leq \frac{C'}{t} \Vert \varphi\Vert_{L^{s}(\mathbb{R}^n)}^s
\end{align*}
for
\begin{align*}
C':= 20^{n-\alpha s} C_1 2^{s+1},
\end{align*}
which verifies \eqref{621.6}.

Now we are left to show that there exists $C''=C''(n, \alpha,s)>0$ such that
\begin{align*}
    {\rm cap}_{\alpha,s}(V_t) \leq \frac{C''}{t} \Vert \varphi \Vert_{L^s(\mathbb{R}^n)}^s.
\end{align*}
To this end, we let $x \in V_t$ and observe that for $y_1,$ $y_2 \in B(x, r_x)$ and $z \in \mathbb{R}^n \setminus B(x,  4 r_x)$, we have
\begin{align*}
|y_{1}-z|\geq  \frac{1}{2} |y_{2}-z|,
\end{align*}
and hence 
\begin{align*}
  I_{\alpha}(y_{1}-z)=\frac{1}{\gamma(\alpha)}\frac{1}{|y_{1}-z|^{n-\alpha}}\leq \frac{1}{\gamma(\alpha)}\frac{2^{n-\alpha}}{|y_{2}-z|^{n-\alpha}}=2^{n-\alpha} I_{\alpha}(y_{2}-z).
\end{align*}
Therefore, for any $y_{1}\in B(x,r_x)$, we have
\begin{align*}
I_{\alpha}\ast\psi_x(y_{1})\leq 2^{n-\alpha}\inf_{y\in B(x,r_x)}I_{\alpha}\ast\psi_x(y)\leq 2^{n-\alpha}\inf_{y \in B(x,r_x)}I_{\alpha}\ast\varphi(y),
\end{align*}
where we have used the preceding inequality and non-negativity of $\varphi$.  The combination of this inequality and the defining condition in \eqref{623} yields the estimate
\begin{align*}
\frac{t}{2^{s+1}} &< \frac{1}{c_{\alpha,s}r_x^{n-\alpha s}} \int_{B(x,r_x)} \left( I_\alpha \ast \psi_x \right)^s \; d {\rm cap}_{\alpha,s}\\
&\leq \frac{2^{n-\alpha}}{c_{\alpha,s} r_x^{n-\alpha s}} \int_{B(x, r_x)} \left(\inf_{y \in B(x, r_x)} I_\alpha \ast \varphi(y)  \right)^s \; d {\rm cap}_{\alpha,s} \\
&= \frac{2^{n-\alpha}}{c_{\alpha,s} r_x^{n-\alpha s}} {\rm cap}_{\alpha,s} \left( B(x,r_x) \right)  \left(\inf_{y \in B(x, r_x)} I_\alpha \ast \varphi(y)  \right)^s \\
&= 2^{n-\alpha} \left(\inf_{y \in B(x, r_x)} I_\alpha \ast \varphi(y) \right)^s\\
&\leq  2^{n-\alpha} \left( I_\alpha \ast \varphi(x) \right)^s
\end{align*}
In particular,
\begin{align*}
   t^{\frac{1}{s}}\leq 2^{(n-\alpha+s+1)/s} I_{\alpha}\ast\varphi(x) \quad \text{ for all } x \in V_t,
\end{align*}
which means that the function
\begin{align*}
\frac{2^{(n-\alpha+s+1)/s} \varphi}{t^{1/s}}
\end{align*}
is admissible for the computation of the capacity of $V_t$, i.e. 
\begin{align*}
    {\rm cap}_{\alpha,s} (V_t) &\leq   \left\Vert \frac{2^{(n-\alpha+s+1)/s} \varphi}{t^{1/s}} \right\Vert_{L^s (\mathbb{R}^n)}^s\\
    &=\frac{C''}{t} \Vert  \varphi \Vert_{L^s (\mathbb{R}^n)}^s
\end{align*}
for 
\begin{align*}
 C'':=   2^{n-\alpha+s+1}.
\end{align*}
This completes the proof with the choice
\begin{align*}
C:= C'+C'' \equiv 20^{n-\alpha s} C_1 2^{s+1}+2^{n-\alpha+s+1}.
\end{align*}
\end{proof}

\begin{proof}[Proof of Theorem \ref{main 1}]

We first prove $(ii)$: Let $\varphi\geq 0$ be such that $I_{\alpha}\ast\varphi\geq|f|^{1/s}$. Lemma \ref{main pro} entails that there exists a constant $C= C (n, \alpha, s) >0$ such that
\begin{align*}
&{\rm cap}_{\alpha,s}\left(\left\{x\in\mathbb{R}^{n}:\mathcal{M}_{{\rm cap}_{\alpha,s}}f(x)>t\right\}\right)\\
&\leq {\rm cap}_{\alpha,s}\left(\left\{x\in\mathbb{R}^{n}:\mathcal{M}_{{\rm cap}_{\alpha,s}}\left((I_{\alpha}\ast\varphi)^{s}\right)(x)>t\right\}\right)\\
&\leq \frac{C}{t}\|\varphi\|_{L^{s}(\mathbb{R}^{n})}^{s}.
\end{align*}
Taking the infimum over such $\varphi$ yields the inequality
\begin{align*}
&{\rm cap}_{\alpha,s}\left(\left\{x\in\mathbb{R}^{n}:\mathcal{M}_{{\rm cap}_{\alpha,s}}f(x)>t\right\}\right) \leq \frac{C}{t}\Gamma_{\alpha,s}(|f|^{1/s}),
\end{align*}
for $\Gamma$ as defined in Lemma \ref{fill_in_the_gaps}.  An application of Lemma \ref{fill_in_the_gaps} then implies
that
\begin{align*}
{\rm cap}_{\alpha,s}\left(\left\{x\in\mathbb{R}^{n}:\mathcal{M}_{{\rm cap}_{\alpha,s}}f(x)>t\right\}\right)\leq \frac{C'}{t}\int_{\mathbb{R}^{n}}|f|\;d{\rm cap}_{\alpha,s}
\end{align*}
for
\begin{align*}
C':= C \frac{2^s}{1-2^{-s}},   
\end{align*}
which is the claim of $(ii)$.

 We next prove $(i)$:  The inequality \begin{align*}
      \frac{1}{r^{n-\alpha s}}\int_{B(x,r)} |f_1 +f_2| \; d{\rm cap}_{\alpha,s} \leq2 \left(  \frac{1}{r^{n-\alpha s}} \int_{B(x,r)} |f_1 |  \; d{\rm cap}_{\alpha,s}+  \frac{1}{r^{n-\alpha s}} \int_{B(x,r)}  |f_2 | \; d{\rm cap}_{\alpha,s} \right)
  \end{align*}
 for every ball $B(x,r)$ in $\mathbb{R}^n$ implies
\begin{align*}
   \left| \mathcal{M}_{{\rm cap}_{\alpha,s}} (f_1 + f_2) (x) \right| &  \leq 2 \left( \left| \mathcal{M}_{{\rm cap}_{\alpha,s}} f_1 (x) \right|  + \left| \mathcal{M}_{{\rm cap}_{\alpha,s}} f_2 (x)  \right| \right).
\end{align*}
We note that the weak-type $(1,1)$ estimate of the operator $\mathcal{M}_{{\rm cap}_{\alpha,s}}$ is given by $(ii)$, and as in the proof of Theorem \ref{maintheorem1} we have
\begin{align*}
    \Vert \mathcal{M}_{{\rm cap}_{\alpha,s}} f \Vert_{L^\infty ({\rm cap}_{\alpha,s})} \leq \Vert  f \Vert_{L^\infty({\rm cap}_{\alpha,s})}.
\end{align*}
Therefore an application of Lemma \ref{interpolationofcapacity} with $T = \mathcal{M}_{{\rm cap}_{\alpha,s}}$ and $C= {\rm cap}_{\alpha,s}$ yields that for every $p>1$, there exists a constant $A_3$ depending on $\alpha$, $s$, $n$ and $p$ such that
\begin{align*}
    \|\mathcal{M}_{{\rm cap}_{\alpha,s}}  f \|_{L^p ({\rm cap}_{\alpha,s})} \leq A_3 \Vert f \Vert_{L^p({\rm cap}_{\alpha,s})}
\end{align*}
and the proof of $(i)$ is complete.
\end{proof}

We next commence with the 
\begin{proof}[Proof of Corollary \ref{corollary_LDT_Riesz}]
Let $f \in L^p({\rm cap}_{\alpha,s})$.  For any $g \in C_c(\mathbb{R}^n)$, one has
\begin{align*}
\lim_{r \to 0} \frac{1}{{\rm cap}_{\alpha,s}(B(0,1))r^{n-\alpha s}}\int_{B(x,r)}|g-g^*(x)|^p \;d\text{cap}_{\alpha,s} &= 0,\\
\frac{1}{{\rm cap}_{\alpha,s}(B(0,1))r^{n-\alpha s}}\int_{B(x,r)}|g^*(x)-f^*(x)|^p \;d\text{cap}_{\alpha,s}&=|g^*(x)-f^*(x)|^p,\\
\limsup_{r \to 0} \frac{1}{{\rm cap}_{\alpha,s}(B(0,1))r^{n-\alpha s}}\int_{B(x,r)}|f-g|^p \;d\text{cap}_{\alpha,s}  &\leq \mathcal{M}_{{\rm cap}_{\alpha,s}} |f-g|^p
\end{align*}
and therefore
\begin{align*}
\limsup_{r\to 0}\frac{1}{{\rm cap}_{\alpha,s}(B(0,1))r^{n-\alpha s}}\int_{B(x,r)}&|f-f^*(x)|^p \;d\text{cap}_{\alpha,s} \\
&\leq C\left(\mathcal{M}_{{\rm cap}_{\alpha,s}} |f-g|^p +|g^*(x)-f^*(x)|^p\right).
\end{align*}
Thus for any $\lambda>0$, using subadditivity of the capacity, the weak-type estimate obtained in Theorem \ref{main 1}, and Chebychev's inequality, we find
\begin{align*}
\text{cap}_{\alpha,s}\left( \left\{ \limsup_{r\to 0}\frac{1}{{\rm cap}_{\alpha,s}(B(0,1))r^{n-\alpha s}}\int_{B(x,r)}|f-f^*(x)|^p>\lambda\right\} \right) \leq \frac{C'}{\lambda^p} \|f-g\|_{L^p({\rm cap}_{\alpha,s})}^p.
\end{align*}
By density of $C_c(\mathbb{R}^n)$ in $L^p({\rm cap}_{\alpha,s})$ we conclude that for any $\lambda>0$
\begin{align*}
    \text{cap}_{\alpha,s}\left( \left\{ \limsup_{r\to 0}\frac{1}{{\rm cap}_{\alpha,s}(B(0,1))r^{n-\alpha s}}\int_{B(x,r)}|f-f^*(x)|^p>\lambda\right\} \right)=0.
\end{align*}
But then writing
\begin{align*}
    &\left\{ \limsup_{r\to 0}\frac{1}{{\rm cap}_{\alpha,s}(B(0,1))r^{n-\alpha s}}\int_{B(x,r)}|f-f^*(x)|^p>0\right\} \\
    &= \bigcup_{n \in \mathbb{N}} \left\{ \limsup_{r\to 0}\frac{1}{{\rm cap}_{\alpha,s}(B(0,1))r^{n-\alpha s}}\int_{B(x,r)}|f-f^*(x)|^p>\frac{1}{n}\right\},
\end{align*}
by countable subadditivity of the capacity we find that
\begin{align*}
   \text{cap}_{\alpha,s}\left( \left\{ \limsup_{r\to 0}\frac{1}{{\rm cap}_{\alpha,s}(B(0,1))r^{n-\alpha s}}\int_{B(x,r)}|f-f^*(x)|^p>0\right\} \right)=0,
\end{align*}
which along with the fact that $f^*$ is $\text{cap}_{\alpha,s}$ quasieverywhere equal to the limit of the averages (the Bessel analog of this result is \cite[Theorem 6.2.1 on p.~159]{AH}) completes the proof.
\end{proof}

\begin{proof}[Proof of Corollary \ref{corollary2}]
The proof is the same as Corollary \ref{maintheorem2}, with Lemma \ref{Comparisonbetweencapmaximal} and Theorem \ref{main 1} in place of Lemma \ref{maximaldimetionalchange} and Theorem \ref{maintheorem1}.  We leave the details to the interested reader.
\end{proof}

\section*{Acknowledgments}
Y.-W. Chen is supported by the National Science and Technology Council of Taiwan under research grant number 110-2115-M-003-020-MY3.  D. Spector is supported by the National Science and Technology Council of Taiwan under research grant number 110-2115-M-003-020-MY3 and the Taiwan Ministry of Education under the Yushan Fellow Program. 

\section*{Conflicts of Interest}
The authors have no conflicts of interest to declare.

\section*{Data Availability Statement}
Data sharing not applicable to this article as no datasets were generated or analyzed during the current study.

\end{document}